\documentclass[11pt,reqno]{amsart}
\usepackage[T1]{fontenc}
\usepackage[utf8]{inputenc}
\usepackage{lmodern}
\usepackage{amsmath,amssymb,mathtools,mathrsfs}
\usepackage[a4paper,margin=27mm]{geometry}
\usepackage{microtype}
\usepackage{enumitem}
\usepackage[hidelinks]{hyperref}
\hypersetup{pdftitle={Finite power-image families and automorphism actions in linear groups},pdfauthor={Zhexin Jiang},pdfsubject={Research draft: power-map images and Kourovka Problem 21.40}}
\newtheoremstyle{upright}{6pt}{6pt}{\normalfont}{}{\bfseries}{.}{.5em}{}
\theoremstyle{upright}
\newtheorem{theorem}{Theorem}[section]
\newtheorem{lemma}[theorem]{Lemma}
\newtheorem{definition}[theorem]{Definition}
\newtheorem{proposition}[theorem]{Proposition}
\newtheorem{corollary}[theorem]{Corollary}
\newtheorem{remark}[theorem]{Remark}
\newtheorem{example}[theorem]{Example}
\newtheorem{maintheorem}{Theorem}

\DeclareMathOperator{\Aut}{Aut}
\DeclareMathOperator{\GL}{GL}
\DeclareMathOperator{\UT}{UT}

\DeclareMathOperator{\Lie}{Lie}
\DeclareMathOperator{\lcm}{lcm}
\DeclareMathOperator{\cl}{class}
\newcommand{\Q}{\mathbb Q}
\newcommand{\Z}{\mathbb Z}
\newcommand{\C}{\mathbb C}
\newcommand{\Pfam}{\mathscr P}
\newcommand{\Rf}{\mathrm R_{\!f}}
\newcommand{\zar}{\mathrm{Zar}}
\newcommand{\Qbar}{\overline{\mathbb Q}}
\setlist[enumerate]{label=\textup{(\roman*)},leftmargin=2em,itemsep=3pt,topsep=5pt}
\numberwithin{equation}{section}
\title[Finite power-image families in linear groups]{Finite power-image families and automorphism actions in linear groups}
\author{Zhexin Jiang}

\address{
	Department of Mathematics,
	School of Science and Engineering,
	The Chinese University of Hong Kong, Shenzhen,
	GuangDong, China
}

\email{zhexinjiang@link.cuhk.edu.cn}
\date{ 15 October 2026}
\subjclass[2020]{
	Primary 20H20;
	Secondary 20E34, 20E36, 20F18, 20F40
}
\keywords{Linear group, automorphism orbit, power map, unipotent group, nilpotent action}
\begin{document}
\begin{abstract}
Let $k/\Q$ be a finitely generated field extension. We characterize the subgroups $G\leq\GL_n(k)$ for which the
power images $P_m(G)=\{g^m:g\in G\}$ form a finite family:this holds precisely when $G=N\rtimes F$, where $F$ is finite
and $N$ is a uniquely divisible subgroup consisting of unipotent matrices.
The subgroup $N$ is canonical and fully invariant.
For $e=\exp(G/N)$, we obtain
$P_m(G)=P_{\gcd(m,e)}(G)$, exactly $\tau(e)$ distinct power
images, and least positive period $e$, where $\tau(e)$ is the
number of positive divisors of $e$.
Over $\Q$, the group coincides with the full group of rational
points of its Zariski closure.
Since finite automorphism-orbit count implies finiteness of
the power-image family, an affirmative answer to Problem~21.40
of the Kourovka Notebook\cite{Kourovka} follows as a corollary.
We further prove that a rational linear group $G$ admits a
decomposition $G\cong(\Q^d,+)\rtimes F$, with $d\geq0$ finite
and $F$ finite, whenever a virtually nilpotent subgroup of
$\operatorname{Aut}(G)$ has finitely many orbits on $G$.
The latter argument uses nonconstant invariant rational
functions arising from nonzero Lie brackets to obstruct
dense orbits. A metabelian action with five orbits on the rational Heisenberg
group shows that solubility of the acting group is insufficient.

\end{abstract}
\maketitle

\section{Introduction And Main Results}
The structure of groups with finitely many automorphism orbits has been
studied under a variety of algebraic hypotheses. For a group $G$ and a
subgroup $A\leq\operatorname{Aut}(G)$, consider the natural left action
$\alpha\cdot g=\alpha(g)$ and write
\[
\omega_A(G)=|A\backslash G|,
\qquad
\omega(G)=\omega_{\operatorname{Aut}(G)}(G).
\]
Thus $A\backslash G$ is the set of $A$-orbits on $G$.
All automorphisms and endomorphisms in this paper are abstract group maps
unless explicitly described as algebraic or linear.

A motivating question is Problem~21.40 of the twenty-first edition of
the Kourovka Notebook, proposed by A.~Dantas and E.~de Melo
\cite{Kourovka}: if $G\leq\operatorname{GL}_n(\mathbb Q)$ and
$\omega(G)<\infty$, must $G$ be virtually soluble?

Two recent preprints present proposed affirmative answers to
Problem~21.40. Jayadevan~\cite[Theorem~1.1]{Jayadevan2026} proved if
$G\leq\operatorname{GL}_n(\mathbb Q)$ and $\omega(G)<\infty$, then $G$
has a normal torsion-free nilpotent subgroup $U$ satisfying
$$
\operatorname{class}(U)\leq n-1,
\qquad
[G:U]\leq(2n+1)^d\leq(2n+1)^{n^2},
\qquad
d=\dim_{\mathbb Q}\operatorname{span}_{\mathbb Q}G.
$$
A closely related AI-generated argument for the same bound is
presented in \cite[\S3.39]{CKU2026}. Both arguments use bounded-degree roots of
algebraic numbers to constrain the spectrum and a trace-radical
construction to obtain the finite-index nilpotent subgroup; neither
requires $G$ to be finitely generated. Jayadevan also treats finite trace
sets and obtains a number-field version by restriction of scalars
\cite[Theorem~1.2 and Corollary~4.1]{Jayadevan2026}. In this paper, we give a structural characterization of subgroups
$G\leq\operatorname{GL}_n(k)$, where $k/\mathbb Q$ is finitely generated,
under the strictly weaker assumption that their power-image family is finite.
For $k=\mathbb Q$, an affirmative answer to Problem~21.40 follows as a corollary.
We further prove virtual abelianness for rational linear groups admitting
a virtually nilpotent subgroup of automorphisms with finitely many orbits.
This refinement is not automatic, as the rational Heisenberg group shows,
and its proof connects nonzero Lie brackets with invariant rational functions.

Earlier structural results begin with solubility or virtual nilpotence.
Bastos--Dantas--de Melo~\cite{BDM2020} obtain a radicable nilpotent kernel
with finite complement for soluble groups of finite rank with finitely
many automorphism orbits. Their later work~\cite{BDM2021} starts instead
with virtual nilpotence and obtains a torsion-free radicable nilpotent
characteristic kernel and a torsion complement.
Dantas--de Melo--Kato~\cite{DMK} study mixed-order soluble groups of finite
rank with four or five automorphism orbits.
Dantas--de Sousa~\cite{DS} prove finiteness for finitely generated locally
graded groups with finitely many automorphism orbits and also study
finite-orbit phenomena in Mal'cev completions of free nilpotent groups.
Accordingly, the finitely generated automorphism-orbit consequence of
our structural theorem is not claimed as a separate new result.

For prescribed acting groups, Jabara~\cite{Jabara2007} proves that $G$ is
finite when $A\leq\operatorname{Aut}(G)$ is finitely generated abelian
and $\omega_A(G)<\infty$. Without finite generation of $A$,
Jabara~\cite{Jabara2010} proves virtual abelianness for an arbitrary
abstract group $G$ when $A$ is abelian and has finitely many orbits.
These results provide the context for our second main theorem, which
restricts $G$ to rational linear groups but permits virtually nilpotent
acting subgroups.

The power-image questions of Hrushovski--Kropholler--Lubotzky--Shalev
\cite{HKLS} and Lubotzky--Meiri~\cite{LM} concern powers inside finitely
generated groups. Our setting does not assume finite generation of the
group, and the finiteness condition studied below need not pass to
finitely generated subgroups.

\subsection*{Notations and Main theorems}
For a group $G$ and a positive integer $m$, let
$$
P_m(G)=\{g^m:g\in G\},\qquad\Pfam(G)=\{P_m(G):m\geq1\}.
$$
The notation $P_m(G)$ denotes the \emph{image set} of the power map.
For every endomorphism $\alpha$ of $G$, the identity
$\alpha(g^m)=\alpha(g)^m$ gives $\alpha(P_m(G))\subseteq P_m(G)$. Thus $P_m(G)$ is a union of automorphism orbits and contains the identity.
Consequently,
\begin{equation}\label{eq:orbit-bound}
\omega(G)<\infty\quad\Longrightarrow\quad
|\Pfam(G)|\leq 2^{\omega(G)-1}<\infty.
\end{equation}
Indeed, the identity forms one orbit, and each of the other orbits can be
included or omitted. Different exponents may yield the same set.

\begin{definition}
	A group is \emph{divisible} (or \emph{radicable}) if every equation
	$x^m=g$, $m\geq1$, has a solution. It is \emph{uniquely divisible}
	if that solution is always unique. A subgroup is \emph{fully invariant} if every
	endomorphism of the ambient group maps it into itself. 
\end{definition}
The finite residual
of $G$ is
$$
\Rf(G)=\bigcap\{L\trianglelefteq G:[G:L]<\infty\}.
$$
For a finite group $B$, write $\exp(B)$ for its exponent, with
$\exp(1)=1$, and let $\tau(e)$ be the number of positive divisors of $e$.

\begin{maintheorem}\label{thm:main}
	Let $k/\Q$ be a finitely generated field extension and let
	$G\leq\GL_n(k)$. The following conditions are equivalent:
	\begin{enumerate}
		\item $\Pfam(G)$ is finite;
		\item $G=N\rtimes F$, where $F$ is finite and $N$ is a uniquely divisible
		subgroup consisting of unipotent matrices.
	\end{enumerate}
	When these conditions hold, let $H$ be the Zariski closure of $G$ in
	$\GL_n(\overline{k})$. Then $H^\circ$ is unipotent and
	\begin{equation}\label{eq:canonical}
	N=G\cap H^\circ
	=\{g\in G:g\text{ is unipotent}\}
	=\Rf(G).
	\end{equation}
	The subgroup $N$ is fully invariant, torsion-free and nilpotent, with nilpotency class
	$\cl(N)\leq n-1$. It is also the largest divisible subgroup of $G$.
	If $e=\exp(G/N)$, then
	\begin{equation}\label{eq:profile}
	P_e(G)=N,\qquad P_m(G)=P_{\gcd(m,e)}(G),\qquad
	|\Pfam(G)|=\tau(e).
	\end{equation}
	The least positive period of the sequence $(P_m(G))_{m\geq1}$ is $e$.
	The trivial group has nilpotency class zero.
\end{maintheorem}

The assumption on the field is relevant to the forward implication.
The reverse implication and the power-image formula follow from a more
general abstract observation about finite extensions of uniquely divisible
groups, proved in Proposition~\ref{prop:profile}.

\begin{corollary}\label{thm:rational}
	Let $G\leq\GL_n(\Q)$ and suppose that $\Pfam(G)$ is finite.
	With $H$ and $N$ as in Theorem~\ref{thm:main}, regarded as a
	$\Q$-defined algebraic group and its canonical subgroup, respectively,
	\[
	N=H^\circ(\Q),\qquad G=H(\Q).
	\]
\end{corollary}

\begin{corollary}\label{cor:kourovka}
	Let $G\leq\GL_n(k)$, where $k/\Q$ is finitely generated.
	If $\omega(G)<\infty$, all conclusions of Theorem~\ref{thm:main} hold,
	and $\omega(N)\leq\omega(G)$. In particular, $G$ is virtually torsion-free
	nilpotent, hence virtually soluble and amenable as a discrete group.
	For $k=\Q$, Theorem~\ref{thm:rational} also applies.
\end{corollary}

Virtual nilpotence cannot in general be strengthened to virtual
abelianness: $\UT_3(\Q)$ is a counterexample, as recalled in
Section~\ref{sec:examples}. A restriction on the acting automorphism group gives the following conclusion.

\begin{maintheorem}\label{thm:abelian-action}
	Let $G\leq\GL_n(\Q)$. Suppose that there is a virtually nilpotent subgroup
	$A\leq\Aut(G)$ such that $\omega_A(G)<\infty$.
	Then the canonical subgroup $N$ in Theorem~\ref{thm:main} is abelian.
	In particular,
	\[
	G\cong(\Q^d,+)\rtimes F
	\]
	for some finite integer $d\geq0$ and some finite group $F$.
	The same abstract conclusion holds when $G$ is linear over a number field.
	Neither $G$ nor $A$ is assumed to be finitely generated.
\end{maintheorem}

The hypothesis on $A$ cannot be weakened to solubility, even to
metabelianness; see Proposition~\ref{prop:metabelian}.

Theorems~\ref{thm:main} and Corollary~\ref{thm:rational} are proved in
Sections~2--3. Section~\ref{sec:nilpotent-actions} proves
Theorem~\ref{thm:abelian-action}, and Section~\ref{sec:examples} gives
examples and limits of the hypotheses.
%All amenability statements refer to discrete groups and are consequences,
%not inputs, of the structural conclusions.

\section{Prime roots And the spectrum}
We first give the field-theoretic facts needed in the proof.

\begin{lemma}\label{lem:binomial}
Let $L$ be a field of characteristic zero, let $p$ be prime, and let
$a\in L^\times$. If $a\notin(L^\times)^p$, then $X^p-a$ is irreducible over $L$.
\end{lemma}
\begin{proof}
If the polynomial were reducible, it would have a root $b$ with
$1\leq d=[L(b):L]<p$. Taking the field norm in $b^p=a$ gives
$\operatorname{N}_{L(b)/L}(b)^p=a^d$.
Choose integers $u,v$ with $ud+vp=1$. Then
\[
 a=\bigl(\operatorname{N}_{L(b)/L}(b)^u a^v\bigr)^p,
\]
a contradiction.
\end{proof}

\begin{lemma}\label{lem:prime-roots}
Let $L/\Q$ be finitely generated and let $a\in L^\times$ be an element
that is not a root of unity. Then $a$ has $p$-th roots for only finitely many
primes $p$ in $L$, i.e., $a\in (L^\times)^p$
\end{lemma}
\begin{proof}
Suppose first that $a$ is algebraic over $\Q$. The relative algebraic
closure $C$ of $\Q$ in $L$ is a number field. Indeed, choose a transcendence
basis $T$ for $L/\Q$ and let $D=[L:\Q(T)]$. Every finite subextension
$C_0/\Q$ of $C/\Q$ satisfies
\[
 [C_0:\Q]=[C_0(T):\Q(T)]\leq D.
\]
The composition of two such subextensions is also a subextension.
Choosing one of maximal degree can shows that $C/\Q$ itself is finite.
If $b^p=a$ with $b\in L$, then $b$ is algebraic over $\Q$, hence lies in $C$.

If some finite-prime valuation $v_{\mathfrak p}(a)$ in $C$ is nonzero,
then $b^p=a$ forces $p\mid v_{\mathfrak p}(a)$ and allows only finitely
many primes. Otherwise $a$ is a unit and so is every such $b$.
Dirichlet's unit theorem~\cite[Chapter~I]{Neukirch} identifies
$\mathcal O_C^\times/\mu(C)$ with a finitely generated free abelian group.
The nonzero class of $a$ is divisible by only finitely many primes.

Now suppose that $a$ is transcendental over $\Q$.
Extend $a$ to a transcendence basis $a,t_2,\ldots,t_r$ and let
$E=\Q(a,t_2,\ldots,t_r)$. In the polynomial ring
$\Q(t_2,\ldots,t_r)[a]$, By using Eisenstein's criterion at the prime $a$, we can shows
that $X^p-a$ is irreducible over $E$. If $b\in L$ satisfies $b^p=a$, then
\[
                 p=[E(b):E]\leq[L:E].
\]
Again only finitely many primes satisfy the inequality.
\end{proof}

\begin{proposition}\label{prop:spectrum}
Let $G\leq\GL_n(k)$, where $k/\Q$ is finitely generated.
If $\Pfam(G)$ is finite, every eigenvalue of every element of $G$ is a root
of unity. Moreover, only finitely many different eigenvalues occur in $G$.
\end{proposition}
\begin{proof}
Assume not, there is a $g\in G$ has an eigenvalue $\lambda$ which is not a root of
unity, and let $L=k(\lambda)$. This is again a finitely generated field.
By Lemma~\ref{lem:prime-roots}, there is an infinite set $S$ of primes
$p>n$ for which $\lambda\notin(L^\times)^p$.
If $p\in S$ and $p\nmid r$,  we have
$\lambda^r\notin(L^\times)^p$ as well. If not, assume $\lambda^r=y^p$ for some $y\in L^\times$, since $p\nmid r$,$\gcd(r,p)=1.$
by Bezout's theorem,  there exist $u,v$ such that $ur+vp=1.$ Then we have $\lambda=\lambda^{ur+vp}=(\lambda^r)^u(\lambda^p)^v.$
Then $\lambda=(y^p)^u(\lambda^p)^v=(y^u\lambda^v)^p.$ So $\lambda\in(L^\times)^p,$ contradiction.

If $g^rin P_p(G)$, $g^r=x^p$ for $x\in G$, then some eigenvalue $\mu$ of $x$ satisfies
$\mu^p=\lambda^r$. Lemma~\ref{lem:binomial} and the characteristic
polynomial of $x$ give the contradictory inequalities
$$p=[L(\mu):L]\leq[k(\mu):k]\leq n.$$ 
Thus $g^r\notin P_p(G)$ whenever $p\in S$ and $p\nmid r$.

Next, we want to construct infinite $P_p(G)$. For distinct $p,q\in S$ we have
$g^p\in P_p(G)\setminus P_q(G)$ since $p\nmid q$. The sets $P_p(G)$ are therefore
pairwise distinct, contrary to the hypothesis  $\Pfam(G)$ is finite.

To obtain uniform finiteness of the spectrum, choose a transcendence basis
$T$ of $k/\Q$ and let $D=[k:\Q(T)]$. If a root of unity $\zeta$ is
as an eigenvalue, then
\[
 [\Q(\zeta):\Q]
 =[\Q(\zeta)(T):\Q(T)]
 \leq[k(\zeta):\Q(T)]\leq nD.
\]
A primitive $r$-th root of unity has degree $\varphi(r)$ over $\Q$.
Only finitely many $r$ satisfy $\varphi(r)\leq nD$, so only finitely many
roots of unity can satisfy the condition.
\end{proof}

\begin{proposition}\label{prop:unipotent}
Under the hypotheses of Proposition~\ref{prop:spectrum}, let
$H=\overline G^{\,\zar}$ and $N_0=G\cap H^\circ$.
Then $H^\circ$ is unipotent, $N_0\trianglelefteq G$ has finite index,
and $N_0$ is torsion-free nilpotent of class at most $n-1$.
Furthermore, $N_0$ is exactly the set of unipotent elements of $G$.
\end{proposition}
\begin{proof}
By Proposition~\ref{prop:spectrum}, only finitely many characteristic
polynomials occur on $G$. Set this finite set as $\Lambda=\{\lambda_1,\ldots,\lambda_s\}.$
For any $g\in G$, the characteristic polynomial of $g$ is $\chi_g(X)=\prod_{j=1}^{n}(X-\mu_j),$ where $\mu_j\in\Lambda.$
Hence $\{\chi_g:g\in G\}$ the possibilities of characteristic polynomials also finite. For any $\chi_h(X)=\det(XI_n-h)=X^n+c_1(h)X^{n-1}+\cdots+c_n(h).$ $c_i(h)$ is a polynomial of $h_{jk}$. Thus, $c_i:H^\circ\to\mathbb A^1$
is regular function on $H^\circ$. They take finitely many values on the
Zariski-dense subgroup $G\cap H^\circ$, hence finitely many values on
$H^\circ$ itself. Since a connected algebraic group is irreducible, each
coefficient function is constant. (if not, then reducible) Its value is determined at the identity. Thus
$$\det(XI_n-h) =\chi_{I_n}(X)=\det(XI_n-I_n)=(X-1)^n\qquad(\forall h\in H^\circ).$$

So $H^\circ$ is unipotent. The triangularization theorem for unipotent
algebraic groups~\cite{Humphreys} places it inside a conjugate of
$\UT_n(\overline{k})= \left\{\begin{pmatrix}
1&*&\cdots&*\\
0&1&\cdots&*\\
\vdots&&\ddots&\vdots\\
0&\cdots&0&1
\end{pmatrix}
\right\}.$. $\UT_n$ is nilpotent and nilpotency class is $n-1$. Consequently $N_0$ is nilpotent of class at most
$n-1$. Nonidentity unipotent matrices in characteristic zero have infinite
order. The group of components $H/H^\circ$ is finite since $H$ has finite components. Considering the natural map $G\to H/H^\circ$, so the kernel is $N_0=G\cap H^\circ$. Identity component is normal subgroup of algebraic group $H^\circ\trianglelefteq H.$, so does $N_0$. We have $G/N_0\hookrightarrow H/H^\circ$, so $N_0$ is finite index normal subgroup in $G$.

Next we want to show every unipotent element $u$ of an algebraic group in
characteristic zero belongs to its identity component. If $u\ne I_n$,
the polynomial map
$$ 
\gamma:\mathbb A^1\to \mathrm{GL}_n\qquad
 t\longmapsto\exp(t\log u)
$$
is an algebraic one-parameter subgroup, with image isomorphic to the
additive group. For any $m\in\Z$, $\gamma(m)=\exp(m\log u)=u^m.$ $\gamma(\Z)\subset H$ since $u\in H$
$\Z$ is Zariski dense in $\mathbb{A}^1$, so does $\gamma(\Z)$in $\gamma(\mathbb A^1)$.  Hence the image lies in the
algebraic group and is connected,$I=\gamma(0)$ and $u=\gamma(1)$. Applying this to $H$ gives
$u\in H^\circ$. Every unipotent element of $H$ belongs to $H^\circ$. Thus, if $u_0\in G$ be unipotent, then $u\in G\cap H^\circ=N_0$, 
$N_0=\{g\in G:g\text{ is unipotent}\}.$
\end{proof}

\section{Stable power images and the canonical divisible subgroup}
For a unipotent matrix $u$, the finite matrix logarithm is
\[
 \log u=\sum_{j=1}^{n-1}\frac{(-1)^{j+1}}{j}(u-I_n)^j.
\]
It is inverse to the finite exponential on nilpotent matrices. In particular,
$\log(u^m)=m\log u$. Thus power maps are injective on any subgroup of
unipotent matrices in characteristic zero. 

\begin{lemma}\label{lem:stable}
Let $G$ be a group with finite $\Pfam(G)$. Suppose that
$N_0\trianglelefteq G$ has finite index $d$ and that every positive-integer
power map on $N_0$ is injective. Then $N_0$ is uniquely divisible and
$N_0=P_d(G)$. In particular, $N_0$ is fully invariant.
\end{lemma}
\begin{proof}
Fix $q\geq1$ and let $a=dq$. The descending chain
$$
 P_a(G)\supseteq P_{a^2}(G)\supseteq P_{a^3}(G)\supseteq\cdots
$$
stabilizes, since all its terms belong to the finite family $\Pfam(G)$.
Suppose $j\geq1$ with $P_{a^j}(G)=P_{a^{j+1}}(G)$.
For $u\in N_0$, we can find a $x\in G$ satisfying
\[
                 u^{a^j}=x^{a^{j+1}}=(x^a)^{a^j}.
\]
Since $x^d\in N_0$, also $x^a\in N_0$. Injectivity on $N_0$ gives
\[
                      u=x^a=(x^d)^q.
\]
Thus every element of $N_0$ has a $q$-th root in $N_0$. Such a root is
unique by hypothesis. Since $q$ was arbitrary, $N_0$ is uniquely divisible.
The inclusion $P_d(G)\subseteq N_0$ follows from the finite quotient, and
the opposite inclusion follows from divisibility of $N_0$.
Every endomorphism of $G$ preserves $P_d(G)$ as a subset.
\end{proof}

We include the splitting argument to separate the stabilization step from
the standard extension theory it subsequently uses.

\begin{lemma}\label{lem:split}
If $N\trianglelefteq G$ is uniquely divisible and nilpotent, and $G/N$
is finite, then $G=N\rtimes F$ for a finite subgroup $F\leq G$.
\end{lemma}
\begin{proof}
In fact, we need to prove the short exact sequence $$ 
1\longrightarrow N
\xrightarrow{\iota}G
\xrightarrow{\pi}B
\longrightarrow1.
 $$ must split under the conditions of the lemma,
where $B=G/N$. i.e., there is a group homomorphism $s:B\longrightarrow G$ satisfies $\pi\circ s=id_B$. 
Let $F=s(B)$, since $B$ is finite and $s$ is injective, so does $F$. For any $g\in G $, we have $g=nf$, where $n\in N, f\in F$,
so we have $G=N\rtimes F$.

Since $N$ is nilpotent, so we have $$1=Z_0(N)\le Z_1(N)\le\cdots\le Z_c(N)=N, 
\text{where \textit{c} is the nilpotency class of }N$$

The case $N=1$ is immediate. We induct on the nilpotency class of $N$.
In the abelian case, $N$ is uniquely divisible, so for any $v\in N$, integer $q$, equation $qw=v$ has unique solution. Then we can define $
\frac pqv=p\left(\frac1qv\right).$ It's easy check $N$ is a $\Q$-vector space. 

We will show any extension of a finite group $B$ by a $\Q[B]$-module
$V$ splits. 
An extension is
represented by an additive $2$-cocycle $c(a,b)$. Set
\[
                 t(a)=\frac1{|B|}\sum_{x\in B}c(a,x).
\]
Summing the identity
$c(a,b)+c(ab,x)=a\cdot c(b,x)+c(a,bx)$ over $x\in B$ gives
$\sum_{x\in B}c(a,b)+\sum_{x\in B}c(ab,x)= |B|c(a,b)+|B|t(ab)$ on the left side, 
$\sum_{x\in B}a\cdot c(b,x)+\sum_{x\in B}c(a,bx)=|B|a\cdot t(b)+|B|t(a).$ on the right side. These give 
$$
                 c(a,b)=t(a)+a\cdot t(b)-t(ab).
$$
In fact, this cocycle is a coboundary, so adjusting the section gives a homomorphic
section. Thus, split.

Now, we prove the non-abelian part. For the induction step we let $Z=Z(N)$. The unique $m$-th root of a central
element is central, because conjugation preserves that root equation.
Thus $Z$ is uniquely divisible. Also $N/Z$ is uniquely divisible. Since $N$ is nonabelian and nilpotent,
the nilpotency class of $N/Z$ is strictly smaller
than that of $N$. We therefore apply the induction hypothesis
to the extension
$$
1\longrightarrow N/Z
\longrightarrow G/Z
\longrightarrow G/N
\longrightarrow 1.
$$
It follows that there exists a finite subgroup $\overline F\leq G/Z$ such that 
$G/Z=(N/Z)\rtimes\overline F.$ Let $q:G\longrightarrow G/Z$ be the quotient map, and define
$E=q^{-1}(\overline F).$ Then $E/Z\cong\overline F,\qquad E\cap N=Z.$
Indeed, if $x\in E\cap N$, then$xZ\in\overline F\cap(N/Z)=1,$
so $x\in Z$. Thus we obtain another short exact sequence
$$
1\longrightarrow Z
\longrightarrow E
\longrightarrow\overline F
\longrightarrow1.
$$ where $Z$ is a uniquely divisible abelian group,
hence a $\mathbb Q[\overline F]$-module,
and $\overline F$ is finite.Because the abelian case proved above, this extension splits.

Consequently, there exists a finite subgroup
$F\leq E$ such that
$$
E=Z\rtimes F,
\qquad
F\cong\overline F.
$$

Since $G/Z=(N/Z)\overline F,$ taking inverse images under $q$ gives
$G=NE=NF.$ Furthermore,$N\cap F\subseteq N\cap E=Z,$ while
$F\cap Z=1$because $F$ is a complement to $Z$ in $E$. Hence $N\cap F=1.$
Since $N\trianglelefteq G$, we conclude that $G=N\rtimes F,$ where $F$ is finite.

\end{proof}

We can use the exponent of the finite quotient to control the $P_m(G)$ for $G$ has 
uniquely divisible normal subgroup $N$ and finite quotient,
Nilpotence of $N$ and splitting of the extension are not required.

\begin{proposition}\label{prop:profile}%
Let $N\trianglelefteq G$ be uniquely divisible and let $B=G/N$ be finite
of exponent $e$.  Let $T(G)$ denote the set of finite-order elements of $G$.
Then
\begin{equation}\label{eq:power-union}
          P_m(G)=\bigcup_{s\in T(G)}s^m C_N(s)\qquad(m\geq1).
\end{equation}
Where $C_N(s)$ is the centralizer of $s$ in $N$. Moreover, $P_e(G)=N$, $P_m(G)=P_{\gcd(m,e)}(G)$,
$|\Pfam(G)|=\tau(e)$, and $e$ is the least positive period of the sequence
of power images.
\end{proposition}
\begin{proof}
A uniquely divisible group is torsion-free. For $g\in G$, let
$u\in N$ be the unique $e$-th root of $g^e$. $B=G/N$ is finite and with index $e$.
For any $g\in G$, $(gN)^e= N$, $g^e\in N$, there exist unique $u\in N$ such that 
$u^e=g^e$, so $(gug^{-1})^e=gu^eg^{-1}=gg^eg^{-1}=g^e.$ $gug^{-1}$ and $u$ are $e$-th roots
of $g^e$ in $N$. $gug^{-1}\in N.$ hence $gug^{-1}=u$. Let $s=gu^{-1}$. Then
$$
                       g=su=us,\qquad s^e=1.
$$
Every finite-order element of $G$ has order dividing $e$: its $e$-th
power lies in the torsion-free group $N$ and has finite order.

For $s\in T(G)$ the subgroup $C_N(s)$ is uniquely divisible, because for $v\in C_N(s)$, there exist unique $w\in N$, such that $w^m=v$,
$(sws^{-1})^m=sw^ms^{-1}=svs^{-1}=v$, $sws^{-1}$ is also a $m$-th root of $v$ in $N$, by uniqueness, $sws^{-1}=w$,
$w\in C_N(s)$.

Next, We want to show $P_m(G)=\bigcup_{s\in T(G)}s^mC_N(s).$
For any $x\in P_m(G).$ there exist $g\in G$, such that $x=g^m.$
we've proved for $g\in G$ can be uniquely write as $g=su=us,$ where $s\in T(G), u\in N.$
Because $s$ and $u$ commute, $u\in C_N(s).$ $g^m=s^mu^m,$ $u^m\in C_N(s).$
Thus $x=s^mu^m\in s^mC_N(s).$ We have $P_m(G)\subseteq\bigcup_{s\in T(G)}s^mC_N(s).$

For any $x\in s^mC_N(s),$ then $x=s^mv$, for any $v\in C_N(s)$. $C_N(s)$ is uniquely divisible,
there exist $w\in C_N(s)$ such that $w^m=v.$ Since $sw=ws$, $(sw)^m=s^mw^m=s^mv=x.$ Then $x\in P_m(G).$
Thus, $\bigcup_{s\in T(G)}s^mC_N(s)\subseteq P_m(G).$ We have  $P_m(G)=\bigcup_{s\in T(G)}s^m C_N(s)(m\geq1).$

It follows immediately that $P_{m+e}(G)=P_m(G)$ and that $P_e(G)=N$.

Let $d=\gcd(m,e)$. There is an integer $a$ prime to $e$ such that
$da\equiv m\pmod e$. Indeed, the class $m/d$ is a unit modulo $e/d$,
and it can be lifted to a unit modulo $e$ by the Chinese remainder theorem.
The map $s\mapsto s^a$ is a bijection of $T(G)$, and
$C_N(s^a)=C_N(s)$. Substitution in~\eqref{eq:power-union} yields
$P_m(G)=P_d(G)$.

Finally, we need to prove the sets $P_d(G)$ for $d\mid e$ are distinct. To see this, pass to the
finite quotient $B$. For each prime $p\mid e$, let $b=v_p(e)$.
The largest $p$-part of the order of an element of $P_d(B)$ is
$p^{b-v_p(d)}.$
The upper bound follows by taking $d$-th powers; equality follows by
choosing an element whose order has $p$-part $p^b$, and then taking a
power to obtain an element of order exactly $p^b$. Such an element exists
by the definition of the exponent; an element of order $e$ is not required. Thus different divisors $d$ give different power images
in $B$, hence in $G$. This proves $|\Pfam(G)|=\tau(e)$.
Finally, if $t>0$ is a period, then $P_{e+t}(B)=P_e(B)=\{1\}$.
Consequently $e\mid e+t$, so $e\mid t$. Hence $e$ is the least period.
\end{proof}

\begin{proof}[Proof of Theorem~\ref{thm:main}]
Assume (i). Proposition~\ref{prop:unipotent} provides
$N_0=G\cap H^\circ$, a finite-index normal unipotent subgroup.
The matrix logarithm makes its power maps injective.
Lemma~\ref{lem:stable} therefore makes it uniquely divisible and fully
invariant. Set $N=N_0$ and apply Lemma~\ref{lem:split} to obtain
$G=N\rtimes F$ with $F$ finite. This proves (ii).
Conversely, (ii) implies (i) by Proposition~\ref{prop:profile}.

Proposition~\ref{prop:unipotent} proves the first two equalities
in~\eqref{eq:canonical} and the nilpotency-class bound. Every homomorphism
from a divisible group to a finite group is trivial: if the target has
exponent $r$, write an element of the source as an $r$-th power.
Hence $N$ lies in every finite-index normal subgroup of $G$.
Since $G/N$ is finite, $N=\Rf(G)$.
The same argument applied to an arbitrary divisible subgroup of $G$
shows that it is contained in $N$. Finally,
Proposition~\ref{prop:profile} proves~\eqref{eq:profile} and the least-period
assertion. The subgroup in (ii) is canonical, although the chosen finite
complement is not claimed to be canonical.
\end{proof}

\begin{proof}[Proof of Corollary~\ref{cor:kourovka}]
Apply~\eqref{eq:orbit-bound} and Theorem~\ref{thm:main}.
Since $N$ is characteristic, the restricted action of $\Aut(G)$ on $N$
has at most $\omega(G)$ orbits. The full group $\Aut(N)$ has at most
that many orbits. Nilpotent groups and their finite extensions are amenable.
\end{proof}

Before we prove the Corollary~\ref{thm:rational}.
We use the following standard form of Mal'cev completion from
Baues--Grunewald~\cite[\S4.2]{BG}.

\begin{lemma}[Mal'cev completion]\label{lem:malcev}
Let $U$ be a unipotent algebraic group defined over $\Q$, and let
$\Gamma\leq U(\Q)$ be finitely generated and Zariski dense.
Then for every $x\in U(\Q)$ there is $m\geq1$ with $x^m\in\Gamma$.
Equivalently, $U(\Q)$ is the radicable hull of $\Gamma$.
\end{lemma}

\begin{lemma}\label{lem:dense-fg}
Every Zariski-dense subgroup $D$ of a unipotent algebraic group $U$ in
characteristic zero contains a finitely generated Zariski-dense subgroup.
\end{lemma}
\begin{proof}
Every closed algebraic subgroup of $U$ is connected: each of its elements
lies in a connected one-parameter subgroup, as in
Proposition~\ref{prop:unipotent}.
Considering all the finitely generated subgroups $\Lambda=\langle d_1,\ldots,d_r\rangle\leq D.$ Their Zariski Closure
satisfies $0\leq\dim\overline\Lambda^{\,\mathrm{Zar}}\leq\dim U.$ Choose $\Gamma\leq D$ be the finitely generated 
group whose Zariski closure has maximal dimension. For any $x\in D$, the closure of $\langle\Gamma,x\rangle$ contains
$\overline\Gamma^{\,\zar}$ and has the same dimension by maximality.
These are connected algebraic groups, so they are equal. Thus
$D\subseteq\overline\Gamma^{\,\zar}$, and density gives
$\overline\Gamma^{\,\zar}=U$.
\end{proof}

\begin{proof}[Proof of Corollary~\ref{thm:rational}]
The Zariski closure $H$ of the subgroup of rational matrices $G$ is
$\Q$-defined; so is its identity component $U=H^\circ$.
The canonical subgroup $N=G\cap U$ is Zariski dense in $U$.
Choose a finitely generated Zariski-dense subgroup $\Gamma\leq N$ by
Lemma~\ref{lem:dense-fg}. If $x\in U(\Q)$, then
$x^m\in\Gamma\subseteq N$ for some $m$ by Lemma~\ref{lem:malcev}.
Unique divisibility of $N$ supplies $y\in N$ with $y^m=x^m$.
Injectivity of powers on unipotent matrices implies $y=x$.
Therefore $N=U(\Q)$.

Write $G=N\rtimes F$ as in Theorem~\ref{thm:main}.$G=\bigcup_{f\in F}Nf.$
Since $F$ is finite, we have 
\[
\begin{aligned}
H
=\overline G^{\,\mathrm{Zar}}
=\overline{\bigcup_{f\in F}Nf}^{\,\mathrm{Zar}}
=\bigcup_{f\in F}\overline{Nf}^{\,\mathrm{Zar}}.
\end{aligned}
\]
right multi $f$ is algebraic isomorphism and Zariski homeomorphism. Thus
$$\overline{Nf}^{\,\mathrm{Zar}}=\overline N^{\,\mathrm{Zar}}f=Uf.$$

So $$H=\bigcup_{f\in F}Uf$$

Each representative $f$ is a rational matrix. Taking rational points
therefore gives $H(\Q)=U(\Q)F=NF=G$.
\end{proof}

\begin{corollary}\label{cor:uniform}
For $n\geq1$, define
\[
             M_n=\lcm\{r\geq1:\varphi(r)\leq n\}.
\]
If $G\leq\GL_n(\Q)$ and $\Pfam(G)$ is finite, then
\[
 \exp(G/N)\mid M_n,\qquad P_{M_n}(G)=N,\qquad
 P_{m+M_n}(G)=P_m(G)\quad(m\geq1).
\]
\end{corollary}
\begin{proof}
A finite-order rational $n\times n$ matrix is diagonalizable over $\C$.
The order of each eigenvalue $\zeta$ satisfies
$\varphi(\operatorname{ord}(\zeta))\leq n$, so the matrix order divides
$M_n$. Apply this to the finite complement $F$, which is isomorphic to
$G/N$, and use Proposition~\ref{prop:profile}.
\end{proof}

\begin{corollary}\label{cor:finite}
Under the hypotheses of Theorem~\ref{thm:main}, $G$ is residually finite
if and only if it is finite. If $G$ is finitely generated, it is finite.
\end{corollary}
\begin{proof}
The first conclusion follows from $\Rf(G)=N$ and the finiteness of $G/N$.
For the second, Schreier's theorem makes $N$ finitely generated.
Its abelianization is a finitely generated divisible abelian group,
hence trivial. A perfect nilpotent group is trivial, so $N=1$.
\end{proof}

\section{Virtual abelianness and nilpotent automorphism actions}
\label{sec:nilpotent-actions}

We begin with an abstract observation that identifies the obstruction
to virtual abelianness in Theorem~\ref{thm:main}.

\begin{lemma}\label{lem:no-finite-index}
A divisible group has no proper finite-index subgroup.
If $N\trianglelefteq G$ is divisible and of finite index, then
\[
            G\text{ is virtually abelian}
            \quad\Longleftrightarrow\quad N\text{ is abelian}.
\]
\end{lemma}
\begin{proof}
Every homomorphism from a divisible group $D$ to a finite group $Q$ is trivial: Considering the group homomorphism $\varphi:D\to Q$,
let $r=\exp(Q)$ for any $x\in D$. By divisible, there exist a $y\in D$, such that $y^r=x$, then $\varphi(x)=\varphi(y^r)=\varphi(y)^r
=1.$ So $\varphi(D)=\{1\}.$
If $[D:L]=s<\infty,$ considering the set of cosets $D/L=\{xL:x\in D\}$, which has $s$ elements. $D$ acts on it by $d\cdot(xL)=(dx)L.$
We get a group homomorphism $\rho:D\longrightarrow S_s.$ $D$ is divisible and $S_s$ is finite, so $\rho$ is trivial. For each $d\in D$,
$dL=L$. Hence $d\in L $ holds for each $d\in D$, $D=L$, divisible group has no proper finite-index subgroup.

If $N\trianglelefteq G, [G:N]<\infty,$ and $N$ is divisible.
If $N$ is abelian, it's finite abelian subgroup of $G$, so $G$ is virtually abelian. 
Conversely, if $G$ is virtually abelian ,then there exist an abelian subgroup
$D\leq G, [G:D]<\infty.$ Then $[N:N\cap D]\leq [G:D]<\infty$ by the injective $N/(N\cap D)\longrightarrow G/D,
n(N\cap D)\longmapsto nD.$
$N$ is divisible, it has no proper finite index subgroup, so $N\cap D=N,$ $N\leq D$, $N$ is also abelian. 

\end{proof}

The next lemma is the algebraic ingredient. The Lie algebra in its
statement is not assumed to be nilpotent.

\begin{lemma}[A dense-orbit obstruction]\label{lem:dense-obstruction}
Let $K$ be an algebraically closed field of characteristic zero, and let
$\mathfrak l$ be a finite-dimensional Lie algebra over $K$.
Suppose that a connected nilpotent linear algebraic group $B$ acts
rationally on $\mathfrak l$ by Lie algebra automorphisms.
If $\mathfrak l$ is nonabelian, then its rational function field contains
a nonconstant $B$-invariant function.
Consequently, if $B$ has a Zariski-dense orbit in $\mathfrak l$, then
$\mathfrak l$ is abelian.
\end{lemma}
\begin{proof}
A connected nilpotent algebraic group in characteristic zero decomposes as
$$
                         B=T\times W,
$$
where $T$ is a algebraic torus and $W$ is connected and unipotent, the characteristic of algebraic torus
$T$ is $\chi:T\longrightarrow K^\times.$;
see~\cite[Corollary~16.48]{Milne}.
Write the weight-space decomposition for the $T$-action as
\[
 \mathfrak l=\bigoplus_{\chi\in S}\mathfrak l_\chi,\qquad
 \mathfrak l_\chi
 =\{v:t\cdot v=\chi(t)v\text{ for all }t\in T\}.
\]
Here $S$ is the finite set of characters $\chi:T\to K^\times$
that occur in the representation; see~\cite[Chapter~12]{Milne}.
Because $T$ centralizes $W$, every $\mathfrak l_\chi$ is $W$-stable.

A nonzero representation of a unipotent group has a nonzero fixed vector
\cite[Section~14b]{Milne}. Applied to $\mathfrak l_\chi^*=\operatorname{Hom}_K(\mathfrak l_\chi,K)$ this yields
a nonzero $W$-invariant linear form $f_\chi$ on $\mathfrak l_\chi$. Define the action of $W$ on the dual space:
$(w\cdot f)(v)=f(w^{-1}\cdot v).$
If $f_\chi$ is invariant, then $f_\chi(w^{-1}\cdot v)=f_\chi(v).$ is invariant vector.
Changing $w$ to $w^{-1}$, we have $f_\chi(w\cdot v)=f_\chi(v).$
Now $f_\chi$ only defined on  $\mathfrak l_\chi$. For any $v\in\mathfrak l$
has unique decomposition 
$$
v=\sum_{\psi\in S}v_\psi,
\qquad v_\psi\in\mathfrak l_\psi.
$$
Extend $f_\chi$ to the whole $\mathfrak l$, define $f_\chi(v)=f_\chi(v_\chi).$  For $v\in\mathfrak l_\chi$, we have $t\cdot v=\chi(t)v.$ $f_\chi(t\cdot v)=\chi(t)f_\chi(v).$ and since
For \(\mathfrak l\) ,
\begin{equation}\label{eq:semi-invariant}
 f_\chi((tw)\cdot v)=\chi(t)f_\chi(v)
 \qquad(t\in T,\ w\in W).
\end{equation}
This is an evaluation formula; it avoids any convention about the
contragredient action on functions.

Suppose that $\mathfrak l$ is nonabelian. Some $\chi,\psi\in S$ satisfy
$[\mathfrak l_\chi,\mathfrak l_\psi]\ne0$.
Since $T$ preserves the bracket,
$
       [\mathfrak l_\chi,\mathfrak l_\psi]
       \subseteq\mathfrak l_{\chi\psi}.
$

Thus $\theta=\chi\psi$ also belongs to $S$.
On the nonempty open set $f_\theta\ne0$, define
$$
                       R=\frac{f_\chi f_\psi}{f_\theta}.
$$
Equation~\eqref{eq:semi-invariant} and $\theta=\chi\psi$ show that
$R((tw)\cdot v)=R(v)$. Hence $R$ is $B$-invariant.
It is nonzero, since the coordinate ring of a vector space is a domain.
Moreover,
$$
                      R(cv)=cR(v)\qquad(c\in K^\times).
$$
It therefore cannot be constant. This remains valid when some of
$\chi,\psi,\theta$ coincide and factors cancel in the quotient.

Finally, suppose that $B$ has a dense orbit, $\mathcal O=B\cdot v_0,
\overline{\mathcal O}^{\,\mathrm{Zar}}=\mathfrak l.$

We have $R=\frac{f_\chi f_\psi}{f_\theta}.$ Because $\Omega=\{f_\theta\neq0\}$ is non empty open set, 
$\mathcal O$ is dense, $\mathcal O\cap\Omega\neq\varnothing.$ And $\Omega$ is invariant under $B$, since 
$f_\theta((tw)\cdot v)=\theta(t)f_\theta(v),$ where $\theta(t)\neq0$. Hence, $\mathcal O\subseteq\Omega.$
Take $v\in\mathcal O$, let $\lambda=R(v).$ Since $R$ is invariant, for all $b\in B$, $R(b\cdot v)=R(v)=\lambda.$
So $R=\lambda$ on the $\mathcal O$. On $\mathcal O$ , $\frac{f_\chi f_\psi}{f_\theta}=\lambda,$ $F=f_\chi f_\psi-\lambda f_\theta$
is zero on the $\mathcal{O}$, \(\mathcal O\) is Zariski dense, so $F$ should be zero on the $\mathfrak l$:$f_\chi f_\psi=\lambda f_\theta.$
The rational function $R=\lambda$ is contradict with $R$ is non-constant.

\end{proof}

\begin{proposition}\label{prop:lie-actions}
Let $\mathfrak l$ be a finite-dimensional Lie algebra over $\Q$.
If a virtually nilpotent subgroup
$\Gamma\leq\Aut_{\Q}(\mathfrak l)$ has finitely many orbits on
$\mathfrak l$, then $\mathfrak l$ is abelian.
\end{proposition}
\begin{proof}
The zero-dimensional case is immediate.
Choose a finite-index nilpotent subgroup $\Gamma_0\leq\Gamma$.
An orbit under a group splits into at most its subgroup index many
orbits under a finite-index subgroup. Thus $\Gamma_0$ still has finitely
many orbits on $\mathfrak l$.

Set $V=\mathfrak l\otimes_{\Q}\Qbar$, and let $C$ be the Zariski
closure of $\Gamma_0$ in $\Aut_{\Qbar}(V)$.
The latter is a closed algebraic subgroup of $\GL(V)$: preserving the
bracket is expressed by polynomial equations in matrix entries.
If $\Gamma_0$ has nilpotency class at most $c$, its $(c+1)$-fold
commutator identities extend to $C$ by Zariski density, since word
maps on algebraic groups are regular. Hence $C$ is nilpotent.

Let $B=C^\circ$ and $\Gamma_1=\Gamma_0\cap B$.
The component group is finite, so $[\Gamma_0:\Gamma_1]<\infty$ and
$\Gamma_1$ still has finitely many orbits. Choose representatives
$v_1,\ldots,v_s\in\mathfrak l$. Then
\[
 \mathfrak l=\bigcup_{i=1}^s\Gamma_1\cdot v_i
            \subseteq\bigcup_{i=1}^s B\cdot v_i.
\]
The rational vector space $\mathfrak l$ is Zariski dense in $V$.
Therefore
\[
                  V=\bigcup_{i=1}^s
                     \overline{B\cdot v_i}^{\,\zar}.
\]
Since $V$ is irreducible, at least one $B\cdot v_i$ is dense.
Lemma~\ref{lem:dense-obstruction} implies $[V,V]=0$, hence
$[\mathfrak l,\mathfrak l]=0$.
\end{proof}

\begin{proof}[Proof of Theorem~\ref{thm:abelian-action}]
Since $\omega(G)\leq\omega_A(G)<\infty$, Theorems~\ref{thm:main}
and~\ref{thm:rational} give
\[
          G=N\rtimes F,\qquad N=U(\Q),\qquad U=H^\circ,
\]
where $U$ is a unipotent $\Q$-group and $N$ is characteristic in $G$.
Thus $A$ acts on $N$ with finitely many orbits.

Let $\mathfrak n=\Lie(U)(\Q)$. The exponential and logarithm identify
$N$ with the finite-dimensional rational nilpotent Lie algebra
$\mathfrak n$, equipped with its Baker--Campbell--Hausdorff group law:$
X*Y=\log(\exp X\exp Y)=X+Y+\frac12[X,Y]+\frac1{12}[X,[X,Y]]+\frac1{12}[Y,[Y,X]]+\cdots..$
The Mal'cev correspondence is functorial for abstract homomorphisms
between uniquely divisible nilpotent groups; in particular,
\[
 \alpha\longmapsto\bigl(X\longmapsto
            \log(\alpha(\exp X))\bigr)
\]
is an isomorphism from $\Aut(N)$ to $\Aut_{\Q}(\mathfrak n)$;
see~\cite{Assmann} for the group--Lie algebra correspondence and
\cite[Theorem~14.37]{Milne} for its algebraic-group form.
Thus no assumption of linearity of the original abstract automorphisms
is being imposed.

The image of $A$ in $\Aut_{\Q}(\mathfrak n)$ is virtually nilpotent
and has finitely many orbits. Proposition~\ref{prop:lie-actions}
therefore gives $[\mathfrak n,\mathfrak n]=0$.
The Baker--Campbell--Hausdorff law reduces to addition, and hence
$N\cong(\Q^d,+)$, where $d=\dim_{\Q}\mathfrak n<\infty$.
The existing splitting gives the required description of $G$.

Finally, if $G\leq\GL_n(K)$ for a number field $K$, restriction of
scalars embeds the same abstract group in
$\GL_{n[K:\Q]}(\Q)$. The given abstract action of $A$ and its orbits
are unchanged, so the rational case applies.
\end{proof}

\begin{remark}\label{rem:action-scope}
The essential restriction on the action is finite orbit count for a
virtually nilpotent subgroup of the \emph{abstract} automorphism group.
Finiteness for the full automorphism group alone is insufficient.
The passage from rational orbits to algebraic orbits above establishes
one dense orbit, not a bijection of orbit spaces.
For more general finitely generated fields, the proof does not
immediately apply: a divisible unipotent subgroup can have infinite
dimension as a rational Mal'cev Lie algebra, and its abstract
automorphisms need not be linear over the larger field.
\end{remark}

\section{Examples and limits of the hypotheses}\label{sec:examples}

The additional condition $\omega(G)<\infty$ is not characterized by
Theorem~\ref{thm:main}; the theorem characterizes the strictly weaker
condition $|\Pfam(G)|<\infty$.

\begin{example}\label{ex:examples}
The group
\[
 U(\Q)=\left\{\begin{pmatrix}1&t\\0&1\end{pmatrix}:t\in\Q\right\}
\]
is isomorphic to $(\Q,+)$, is uniquely divisible, and has two automorphism
orbits. All its power images coincide with the group.
For the affine group $G=\Q\rtimes C_2$, where the involution acts by
inversion, the three automorphism orbits are the identity, the nonzero
translations, and the involutions. Here $e=2$ and
\[
               P_m(G)=\begin{cases}G,&m\text{ odd},\\
                                    \Q,&m\text{ even}.
                       \end{cases}
\]
For example, in coordinates $(\varepsilon,t)$ with
$\varepsilon\in\{1,-1\}$, the automorphisms which scale translations by
$c\in\Q^\times$ and send a chosen involution $r$ to $u(d)r$, $d\in\Q$,
act transitively on the last two sets separately.

In contrast, $\UT_6(\Q)$ is uniquely divisible, so it has just one
power-image set, but has infinitely many automorphism orbits by
\cite[Theorem~1.2]{MK}. Thus finiteness of power images is strictly weaker
than finiteness of automorphism orbits, even for rational unipotent groups.
\end{example}

\begin{example}[The rational Heisenberg group]\label{ex:heisenberg}
Let $\mathcal H=\UT_3(\Q)$, with symmetric coordinates
\[
 h(x,y,t)=
 \begin{pmatrix}
  1&x&t+xy/2\\
  0&1&y\\
  0&0&1
 \end{pmatrix}
 \qquad(x,y,t\in\Q).
\]
Writing $v=(x,y)^{\mathsf T}$ and
$\Omega(v,w)=xy'-yx'$ for $w=(x',y')^{\mathsf T}$, we have
\[
 h(v,t)h(w,s)=h(v+w,t+s+\tfrac12\Omega(v,w)).
\]
The group commutator convention is $[a,b]=aba^{-1}b^{-1}$. Thus
\[
 [h(v,t),h(w,s)]=h(0,\Omega(v,w)),\qquad
 Z(\mathcal H)=[\mathcal H,\mathcal H]=\{h(0,t):t\in\Q\}.
\]
In particular, $\mathcal H$ is nonabelian of class two.
Since $h(v,t)^m=h(mv,mt)$, it is uniquely divisible and has no proper
finite-index subgroup by Lemma~\ref{lem:no-finite-index}.
It is therefore not virtually abelian.

For $M\in\GL_2(\Q)$ and $\ell\in\operatorname{Hom}_{\Q}(\Q^2,\Q)$,
$\Phi_{M,\ell}(h(v,t))=h(Mv,(\det M)t+\ell(v))$ is an automorphism, since
$\Omega(Mv,Mw)=\det(M)\Omega(v,w)$.
These maps are transitive on the nonidentity central elements, and also
on the noncentral elements. For the latter assertion, first choose
$M$ with $Mv=w$ for prescribed nonzero $v,w$, and then choose $\ell$
with $\ell(v)=s-(\det M)t$.
The identity and the center are characteristic, so
$\omega(\mathcal H)=3.$
This is the rational case of~\cite[Theorem~3.1]{MK}; the calculation is
included to distinguish full automorphism orbits from the smaller
acting group in the next proposition.
\end{example}

\begin{proposition}[A metabelian action with five orbits]
\label{prop:metabelian}
There is a metabelian subgroup $A\leq\Aut(\mathcal H)$ such that
$\omega_A(\mathcal H)=5$. In particular, the conclusion of
Theorem~\ref{thm:abelian-action} fails if ``virtually nilpotent'' is
replaced by ``metabelian'' or ``soluble''.
\end{proposition}
\begin{proof}
For $a,b\in\Q^\times$ and $p,q\in\Q$, define
\[
 \phi_{a,b,p,q}(h(x,y,t))
       =h(ax,by,ab\,t+px+qy).
\]
These are the automorphisms $\Phi_{M,\ell}$ above with
$M=\operatorname{diag}(a,b)$ and $\ell(x,y)=px+qy$.
They form a group $A$: composition is given by
\[
 \phi_{a,b,p,q}\circ\phi_{c,d,r,s}
       =\phi_{ac,bd,\,ab r+pc,\,ab s+qd}.
\]
The map $\phi_{a,b,p,q}\mapsto(a,b)$ is an epimorphism onto
$(\Q^\times)^2$, with abelian kernel
\[
                 \{\phi_{1,1,p,q}:p,q\in\Q\}\cong(\Q^2,+).
\]
Thus $A$ is metabelian.

The five invariant subsets and representatives are
\[
\begin{array}{c|c}
 \text{condition}&\text{representative}\\ \hline
 x=y=t=0 & h(0,0,0)\\
 x=y=0,\ t\ne0 & h(0,0,1)\\
 x\ne0,\ y=0 & h(1,0,0)\\
 x=0,\ y\ne0 & h(0,1,0)\\
 x\ne0,\ y\ne0 & h(1,1,0)
\end{array}
\]
Each subset is one orbit. When $(x,y)\ne(0,0)$, choose $a,b$ to
normalize the nonzero coordinates and then choose $p,q$ to make
$ab\,t+px+qy=0$. When $x=y=0$ and $t\ne0$, choose $ab=t^{-1}$.
The zero patterns cannot change, so these five orbits are distinct.
Example~\ref{ex:heisenberg} shows that $\mathcal H$ is not virtually
abelian, proving the asserted failure.

For completeness, $A$ is not virtually nilpotent. Let
$$
 D=\phi_{2,1,0,0},\qquad U_q=\phi_{1,1,0,q}.
$$
Then $DU_qD^{-1}=U_{2q}$.
Any finite-index subgroup $A_0\leq A$ contains
$V=\{U_q:q\in\Q\}$, because $V$ is divisible, and it contains
$D^r$ for some positive integer $r$.
Consequently
\[
                  [D^r,U_q]=U_{(2^r-1)q}.
\]
Since multiplication by $2^r-1$ is surjective on $\Q$, every term
of the lower central series of $A_0$ contains $V$.
Thus $A_0$ cannot be nilpotent.
\end{proof}

\begin{corollary}\label{cor:dimension}
If $G\leq\GL_2(\Q)$ and $\Pfam(G)$ is finite, then $G$ is virtually
abelian. For every $n\geq3$, there is a subgroup of $\GL_n(\Q)$ with
three automorphism orbits that is not virtually abelian.
\end{corollary}
\begin{proof}
Theorem~\ref{thm:main} gives $\cl(N)\leq1$ in dimension two.
For $n\geq3$, use $\mathcal H$ and the block embedding
$g\mapsto\operatorname{diag}(g,I_{n-3})$, with the identity block
omitted for $n=3$.
\end{proof}

\begin{remark}[Limits of the field statements]\label{rem:limits}
The equality $G=H(k)$ cannot replace Theorem~\ref{thm:rational} over a
general finitely generated field. Embed the group $U(\Q)$ from
Example~\ref{ex:examples} in $\GL_2(\Q(\sqrt2))$. Its Zariski closure is
the full one-dimensional unipotent group, whose $\Q(\sqrt2)$-points are
strictly larger than $U(\Q)$.

The finite-generation hypothesis on the field cannot be discarded from
the forward implication of Theorem~\ref{thm:main}: every invertible complex
matrix has an $m$-th root for every $m\geq1$, so
$P_m(\GL_2(\C))=\GL_2(\C)$ for every $m$. Nevertheless
$\GL_2(\C)$ is not virtually soluble (its connected Zariski closure is
not soluble). This example concerns the \emph{power-image} hypothesis;
it makes no claim about finite automorphism-orbit count for this group.
\end{remark}

\section*{Acknowledgements}

I am deeply grateful to my supervisor, Professor Hou Yong,
for his continued support, many valuable discussions, and
guidance throughout my studies.
His suggestion to investigate virtual abelianness inspired
Theorem~B.
He also encouraged me to look beyond the resolution of an
individual problem and seek a deeper understanding of the
underlying structures; this encouragement led me to develop
the present work beyond an affirmative answer to Problem~21.40
of the Kourovka Notebook.

This paper is dedicated to the memory of my grandfather, Feng Bengui.

\section*{Declaration of AI Tool Use}
The main mathematical work presented in this paper was completed in July 2026.In April, Gemini 3 Pro suggested the example presented in Example 5.1. After the manuscript was completed, ChatGPT was used to assist in verifying references and related literature. AI tools were also employed to review the details of the proofs and to check grammar.

The main mathematical results, arguments, and exposition are the author's own work. The entire manuscript was written and typeset by the author, rather than generated by AI. All AI-assisted suggestions were independently examined, and the author takes full responsibility for the mathematical correctness, originality, and final content of this paper.

\end{document}